\documentclass[12pt]{article}
\RequirePackage[colorlinks,citecolor=blue,urlcolor=blue,linkcolor=blue]{hyperref}
\hypersetup{
colorlinks = true,
citecolor=blue,
urlcolor=blue,
linkcolor=blue,
pdfauthor = {James Burridge, Alexey Kuznetsov, Mateusz Kwasnicki, Andreas Kyprianou},
pdfkeywords = {subordinator, Kendall identity, explicit transition density, Laplace transform identity, Bessel functions, Lambert W-function, Gamma function, complete monotonicity, generalized gamma convolutions},
pdftitle = {New families of subordinators with explicit transition probability semigroup},
pdfpagemode = UseNone
}
\usepackage{graphicx,xspace,colortbl}
\usepackage{amsmath,amsthm,amsfonts}
\usepackage{color}
\usepackage{fancybox}
\usepackage{epsfig}
\usepackage{subfig}
\usepackage{pdfsync}
    \def\qed{\hfill$\sqcap\kern-8.0pt\hbox{$\sqcup$}$\\}
    \def\beq{\begin{eqnarray}}
    \def\eeq{\end{eqnarray}}
    \def\beqq{\begin{eqnarray*}}
    \def\eeqq{\end{eqnarray*}}

\DeclareMathOperator{\im}{Im}
    \def\p{{\mathbb P}}
    \def\e{{\mathbb E}}
    \def\r{{\mathbb R}}
    \def\c{{\mathbb C}}
    	
    \def\d{{\textnormal d}}

\newtheorem{theorem}{Theorem}

\newtheorem{proposition}{Proposition}

\theoremstyle{definition}

\newtheorem{remark}{Remark}

\title{New families of subordinators with explicit transition probability semigroup}

\author{
{
J. Burridge
\footnote{Department of Mathematics, University of Portsmouth, Lion Gate Building, Lion Terrace, Portsmouth, 
Hampshire, PO1 3HF, U.K.. Email: james.burridge@port.ac.uk}}
,\, 
{
A. Kuznetsov
\footnote{Department of Mathematics and Statistics, 
York University, 
4700 Keele Street, 
Toronto, Ontario, 
M3J 1P3, Canada. Email: kuznetsov@mathstat.yorku.ca}}
,\, 
{
M. Kwa\'snicki
\footnote{
Institute of Mathematics and Computer Science, Wroclaw University of Technology, ul. Wybrzeze Wyspia\'nskiego 27,
50-370 Wroclaw, Poland. Email: mateusz.kwasnicki@pwr.wroc.pl}}
,\, 
{
A. E. Kyprianou
\footnote{Department of Mathematical Sciences, University of Bath, Claverton Down, Bath, BA2 7AY, U.K. Email: a.kyprianou@bath.ac.uk}}
}
\date{\footnotesize This version: \today}

\begin{document}

\maketitle

\begin{abstract}
\bigskip
There exist only a few known examples of subordinators for which the transition probability density can be computed explicitly along side an expression for its L\'evy measure and Laplace exponent. Such examples are useful in several areas of applied probability. For example, they are used in mathematical finance for modeling stochastic time change. They appear in combinatorial probability to construct sampling formulae, which in turn is related to a variety of issues in the theory of coalescence models. Moreover, they have also been extensively used in the potential analysis of subordinated Brownian motion in dimension $d\geq 2$. In this paper, we show that Kendall's classic identity for spectrally negative L\'evy processes can be used to construct new families of subordinators with explicit transition probability semigroups. We describe the properties of these new subordinators and emphasise some interesting connections with  explicit and previously unknown Laplace transform identities and with complete monotonicity properties of certain special functions. 
\end{abstract}

{\vskip 0.5cm}
{\vskip 0.5cm}

\newpage

\section{Introduction}\label{sec_introduction}
Subordinators with explicit transition semigroups have proved to be  objects of broad interest on account of their application in a variety of different fields. We highlight three of them here. The first case of  interest occurs in mathematical finance, where   subordinators are used to perform time-changes of other stochastic processes to model the effect of stochastic volatility in asset prices, see for example \cite{CGMY} and \cite{Cont}.  A second application occurs in the theory of potential analysis of subordinated Brownian motion in high dimensions, which  has undergone significant improvements thanks to the study of a number of key examples, see for example   \cite{Song_Vondracek}  and \cite{KSV}. A third area in which analytic detail of the transition semigroup of a subordinator can lead to new innovations is that of combinatorial stochastic processes. A variety of sampling identities are intimately related to the range of particular subordinators, see for example \cite{Gnedin}. Moreover this can also play an important role in the analysis of certain coalescent processes, see \cite{Pitman}.

In this paper we will use a simple idea based on Kendall's identity for spectrally negative L\'evy processes to construct some new families of subordinators with explicit transition semigroup.  Moreover, we describe their properties, with particular focus on the associated L\'evy measure and Laplace exponent in each of our new examples. The inspiration for the main idea in this paper came about by digging deeper into \cite{Burridge}, where a remarkable identity  appears in the analysis of the relationship between the first passage time of a random walk and the total progeny of  a discrete-time, continuous-state branching process.

The rest of the paper is organised as follows. In the next section we remind the reader of Kendall's identity and thereafter, proceed to our main results. These results give a simple method for generating examples of subordinators with explicit transition semigroups as well as simultaneously gaining access to analytic features of their L\'evy measure and Laplace exponent.  In Section \ref{sec_examples} we put our main results to use in generating completely new examples. Finally, in Section \ref{sec_applications} we present some applications of 
these results to explicit Laplace transform identities and complete monotonicity properties of certain special functions.

\section{Kendall's identity and main results}

Let $\xi$ be a spectrally negative L\'evy process with Laplace exponent defined by 
 \begin{equation}
\psi(z)=\ln \e[\exp(z \xi_1)],\qquad  z\ge 0.
\label{Laplace}
\end{equation}
In general, the exponent $\psi$ takes the form
\[
\psi({z}) = a{z} + \frac{1}{2}\sigma^2{z}^2 + \int_{(-\infty,0)} ({\rm e}^{{z} x} -1 - {z} x \mathbf{1}_{(x> - 1)})\Pi_\xi(\d x)
\]
where $a\in\mathbb{R}$, $\sigma^2\geq 0$ and $\Pi_\xi$ is a measure concentrated on $(-\infty,0)$ that satisfies $\int_{(-\infty, 0)} (1\wedge x^2)\Pi_\xi(\d x)<\infty$, and is called the L\'evy measure.  From this definition, it is easy to deduce that $\psi$
 is convex on $[0,\infty)$, and it satisfies 
$\psi(0)=0$ and $\psi(+\infty)=+\infty$.  Hence, for every $q>0$, there exists a unique solution  $z=\phi(q) \in (0,\infty)$ to 
the equation $\psi(z)=q$. We will define $\phi(0)=\phi(0^+)$. Note that $\phi(0)=0$ if and only if 
$\psi'(0)\ge 0$, which, by a simple differentiation of (\ref{Laplace}), is equivalent to $\e[\xi_1]\ge 0$. 

 Let us define the first passage times  
 \begin{equation}
 \tau_x^+:=\inf\{ t>0 \; : \; \xi_t>x\}, \qquad x\geq 0.
 \label{firstpassage}
 \end{equation}
 It is well-known
 (see Theorem 3.12 and Corollary 3.13 in \cite{Kyprianou}) that 
$\{\tau_x^+\}_{x\ge 0}$ is a subordinator, killed at rate $\phi(0)$, whose Laplace exponent, $\phi(q)$, satisfies
\beqq
\e\left[{\rm e}^{-q \tau_x^+ } {\bf 1}_{\{\tau_x^+ < +\infty\}} \right]={\rm e}^{-x \phi(q)}, \qquad q\geq 0.
\eeqq 
In general, the Laplace exponent $\phi$ is a Bernstein function. In particular, it  takes the form
\begin{equation}
\phi(z) = \kappa + \delta z + \int_{(0,\infty)} (1-{\rm e}^{-z x})\Pi(\d x),
\label{Laplaceexponent}
\end{equation}
for some  $\kappa,\delta\geq 0$ and  measure, $\Pi$, concentrated on $(0,\infty)$, satisfying $\int_{(0,\infty)}(1\wedge x)\Pi(\d x)<\infty$. The constant $\kappa$ is called the killing rate and $\delta$ is called the drift coefficient.

Kendall's identity (see \cite{ECP1038} and Exercise 6.10 in \cite{Kyprianou}) states that
\beq\label{Kendalls_identity}
\int_y^{\infty} \p(\tau_x^+ \le t) \frac{\d x}{x}=\int_0^t \p(\xi_s>y) \frac{\d s}{s}.  
\eeq
If the distribution of $\xi_t$ is absolutely continuous for all $t>0$ 
then the measure $\p(\tau_x^+ \in \d t)$ is also absolutely continuous and has the density
\beq\label{Kendalls_identity_v2}
\p(\tau_x^+ \in \d t)=\frac{x}{t} p_{\xi}(t,x) \d t,\qquad x,t> 0,
\eeq
where $p_{\xi}(t,x)\d x=\p(\xi_t \in \d x)$. 
On the one hand, one may view Kendall's identity as an analytical consequence 
of 
the Wiener-Hopf factorisation for spectrally negative L\'evy processes. On the other, its probabilistic roots are related to certain combinatorial arguments associated to random walks in the spirit of the classical ballot problem.

Kendall's identity gives a very simple way of constructing new subordinators with explicit transition semigroup. Indeed, if we start with a spectrally negative process $\xi$ for which the transition probability density $p_{\xi}(t,x)$ is known, then $\tau_x^+$ is the desired subordinator with the explicit transition density given by \eqref{Kendalls_identity_v2}. One way to build a spectrally negative process with known transition density (as indeed we shall do below) is  as follows: start with a subordinator $X$, which has an explicit transition probability density and then define the spectrally negative process $\xi_t= t-X_t$. This also gives us a spectrally negative process with explicit transition probability density.   The above approach was used in older statistics literature (see \cite{Kendall1957,Letac1990}) in order to  generate new examples of infinitely divisible distributions. Our goal in this paper is to systematically apply this method to create new families of subordinators, describe their L\'evy measure and Laplace exponent, and to study their properties.


 Before stating our main theorem, let us introduce some notation and definitions. We write ${\mathcal N}$ for the class of all subordinators, started from zero, having zero drift and zero killing rate. The Laplace exponent of a subordinator $Y\in {\mathcal N}$ is defined by
$\Phi_Y(z):=-\ln \e\left[\exp(-z Y_1) \right]$, $z\ge 0$. From the L\'evy-Khinchine formula
we know that 
\beq\label{Levy_Khinchine}
\Phi_Y(z)=\int_{(0,\infty)} \left(1-{\rm e}^{-zx}\right) \Pi_Y(\d x), \;\;\; z\ge 0, 
\eeq 
where $\Pi_Y$ is the Levy measure of $Y$. 
When it exists, we will denote the transition probability density function
of $Y$ as $p_Y(t,x):=\frac{\d}{\d x} {\mathbb P}(Y_t \le x)$, $x>0$.

\begin{theorem}\label{theorem_main}
For $X\in {\mathcal N}$ and $q>0$, define $\phi(q)$ as the unique solution to 
\beq\label{eqn_inverse_1}
z-\Phi_X(z)=q.
\eeq
Define $\phi(0)=\phi(0+)$. Then we have the following:
\begin{itemize} 
\item[(i)] The function $\Phi_Y(z):=\phi(z)-\phi(0)-z$ is the Laplace exponent of a subordinator $Y \in {\mathcal N}$. 
\item[(ii)] If the transition semi-group of $X$ is absolutely continuous with respect to Lebesgue measure, then the transition semi-group of $Y$ is given by
\beq\label{formula_pY_pX}
p_Y(t,y)=\frac{t}{t+y} {\rm e}^{\phi(0) t} p_X\left(t+y, y\right), \;\;\; y>0
,
\eeq
and the Levy measure of $Y$ is given by
\beq\label{formula_PiY}
\Pi_Y(\d y)=\frac{1}{y} p_X(y, y)\d y, \;\;\; y>0. 
\eeq
\end{itemize}
\end{theorem}

\begin{proof}
The function $\phi(q)$ defines the Laplace exponent of the subordinator corresponding to the first passage process (\ref{firstpassage}). Moreover, appealing to the standard facts that the   
drift coefficient of $\phi$ is equal to $\lim_{q\to\infty}\phi(q)/q$ and that $\phi(\infty) = \infty$, as well as the fact that $X$ has zero drift, one notes that 
\[
\lim_{q\to\infty}\frac{\phi(q)}{q} = \lim_{q\to\infty}\frac{\phi(q)}{\phi(q)  - \Phi_X(\phi(q))} =  \lim_{q\to\infty}\frac{1}{1 - \Phi(\phi(q))/\phi(q)}=1.
\]
Moreover, noting that $\phi(0)$ is another way of writing the killing rate of the subordinator corresponding to $\phi$,  it follows that  the function
$\phi_Y(z)=\phi(z)-\phi(0)-z$ belongs to the class ${\mathcal N}$. 
Formula \eqref{formula_pY_pX} follows at once from Kendall's identity as it appears in  \eqref{Kendalls_identity_v2}. The formula \eqref{formula_PiY} follows from the fact that 
\beq\label{Levy_measure_as_a_limit}
\Pi_Y(\d x)=\lim\limits_{t\to 0^+} \frac{1}{t} \p(Y_t \in \d x),\qquad x>0,
\eeq
see for example the proof of Theorem 1.2 in  \cite{Bertsub}. 
\end{proof}

In constructing new subordinators, the above theorem has deliberately  eliminated certain scaling parameters. For example, 
one may consider working more generally  with the spectrally negative process 
$\xi_t=\lambda t-X_t$, $t\geq0$ for some $\lambda>0$. However, this can be reduced to the case that $\lambda = 1$ by factoring out the constant $\lambda$ from $\xi$ and noting that $\lambda^{-1}X$ is still a subordinator.

\label{discussion_repeated_Kendall}
Theorem \ref{theorem_main} raises the following natural question concerning its  iterated use. Suppose we have started from a spectrally negative process, say  $\xi^{(1)}$  and have constructed a corresponding subordinator $Y^{(1)}$. Can we take this subordinator, define a new, spectrally negative L\'evy process $\xi^{(2)}_t: = t - Y^{(1)}_t$, $t\geq 0$, and  feed it into back into Theorem \ref{theorem_main} to obtain a new subordinator $Y^{(2)}$? The answer is essentially ``no": one can check that the subordinator $Y^{(2)}$ could also
 be obtained by one application of this procedure starting from the scaled process $\theta \xi_{ct}$ for appropriate constants $\theta, c>0$. In other words, applying the Kendall identity trick twice does not give us fundamentally new processes.

\bigskip

Recent potential analysis of subordinators has showed particular interest in the case of {\it complete subordinators}, following their introduction in  \cite{SV2006}. The class of complete  subordinators can be defined by the analytical structure of their Laplace exponents, which are also known as {\it complete Bernstein functions}.
 In addition to the representation in (\ref{Laplaceexponent}), a function $f$ on $(0, \infty)$ is a complete Bernstein function (CBF in short) if
\[
 f(z) = c_0 + c_1z + \frac{1}{\pi} \int_{(0, \infty)} \frac{z}{z + s} \, \frac{m(\d s)}{s}
\]
for some $c_0, c_1 \ge 0$ and a $\sigma$-finite positive measure $m$ on $(0, \infty)$ satisfying $\int_{(0, \infty)} \min(s^{-1}, s^{-2}) m(\d s) < \infty$.  Equivalently, $f$ is the Laplace exponent of a (possibly killed) subordinator $X$, whose L\'evy measure has a completely monotone density 
\beqq
\pi_X(x)=\int_0^{\infty} {\rm e}^{-x s}m(\d s). 
\eeqq
Let us denote $\c^+:=\{z\in \c \; : \; \im(z)>0\}$ and similarly $\c^-:=\{z\in \c \; : \; \im(z)<0\}$.
It is known (see Theorem 6.2 in \cite{SSV2012}) that CBFs extend to analytic functions that map $\c \setminus (-\infty, 0]$ into $\c \setminus (-\infty, 0]$ and belong to 
the class of {\it Pick functions}, that is functions $g$ analytic in $\c^+$ such that 
$g(z) \in \overline{\c^+}$ for all $z\in \c^+$. Conversely, any Pick function which takes nonnegative real values on $(0, \infty)$ 
is a CBF. For more information on CBFs see \cite{SSV2012}.

Our next result below investigates sufficient conditions on $\xi$ to ensure that the resulting subordinator, $Y_t=\tau_t^+$,  is a complete subordinator. 
\begin{proposition}\label{thm_compl_monotone}
Let $\xi$ be a spectrally negative process with a L\'evy density  $\pi_{\xi}(x)$, $x<0$. 
If $\pi_{\xi}(-x)$ is a completely monotone function, then the subordinator $Y$ has a L\'evy density, say $\pi_Y(x)$, and it is completely monotone. 
\end{proposition}
\begin{proof}
The proof is based on the following result (see Proposition 2 in \cite{Nakamura}): 
{\it If $\Phi$ is a CBF and $\phi$ is the inverse function of the strictly increasing function $z \in (0,\infty) \mapsto z \Phi(z)$, then $\phi$ is also a CBF.}

Let $H$ denote the descending ladder height process for $\xi$ (which is a subordinator, possibly a killed one), and let $\Phi_H$ be its Laplace exponent. Then $\psi(z) = (z - c) \Phi_H(z)$, where $c = \phi(0)$ (see formula (9.1) in \cite{Kyprianou}). 
Theorem 2 in \cite{Rogers1983} tells us that if $\pi_\xi(-x)$ is completely monotone, then $\pi_H(x)$ is completely monotone, and therefore $\Phi_H$ is a CBF. Let $\tilde\Phi_H(z) = \Phi_H(z + c)$ and $\tilde\psi(z) = \psi(z + c)$, so that $\tilde\psi(z) = z \tilde\Phi_H(z)$. Note that $z = \phi(q)$ if and only if $\psi(z) = q$, that is, $\tilde\psi(z - c) = q$. Therefore, $\phi(q) = \tilde{\psi}^{-1}(q) + c$. Since $\tilde\Phi_H(z)$ is a CBF, by the above-mentioned result, $\tilde\psi^{-1}$ is a CBF. It follows that also $\phi$ is a CBF, and therefore $\pi_Y(x)$ is completely monotone.
\end{proof}

\begin{remark}
A curiosity that arises from the above result is that when $\xi$ is a spectrally negative L\'evy process of unbounded variation and has a L\'evy density which is  completely monotone, then it is automatically the case that there is a version of $\xi$'s transition density for which  $p_\xi(t,0)/t$ is completely monotone.  Indeed this follows directly from Kendall's identity and (\ref{Levy_measure_as_a_limit}). Referring to the discussion following Proposition 2.2 in \cite{Bertsub}, 
it follows that the potential density of the inverse local time at zero of $\xi$, which is proportional to $p_\xi(t,0)$, is therefore the product of a linear function and a completely monotone function.
\end{remark}

One corollary of Proposition \ref{thm_compl_monotone} is that the transformation described in Theorem \ref{theorem_main}, which maps a subordinator $X$ into a subordinator $Y_t=\tau_t^+$, preserves the class of complete subordinators. As our next result shows, 
this transformation also preserves an important subclass of complete subordinators. 
We define the class of 
{\it Generalized Gamma Convolutions} (GGC) as the family of infinite divisible distributions on $(0,\infty)$ having L\'evy density
$\pi(x)$, such that the function $x\pi(x)$ is completely monotone. In other words, 
\beqq
x\pi(x) =\int_0^{\infty} {\rm e}^{-xy} U(\d y),
\eeqq
for some $\sigma$-finite and positive measure $U$, which is called {\it Thorin measure}. 
The measure $U$ must satisfy the following integrability condition
\beqq
\int_0^{\infty} \left(|\ln(y)| \wedge \tfrac{1}{y} \right) U(\d y)<\infty
\eeqq
in order for $\Pi(\d x)=\pi(x)\d x$ to be a L\'evy measure of a positive random variable. 
The class of GGC can also be defined as the smallest class of distributions on $(0,\infty)$, which contains all gamma distributions 
and which is closed under convolution and weak convergence. 
See \cite{Bondesson} and \cite{Song_Vondracek} for additional information on the class of GGC and its distributional properties. 
We say that a subordinator $X$ belongs to the Thorin class ${\mathcal T}$ if the distribution of $X_1$ is GGC.
The family ${\mathcal T}_0$ is defined as the subclass of all subordinators in ${\mathcal T}$ which  
have zero linear drift.

\begin{proposition}\label{prop_Thorin_class}
Assume that $X\in {\mathcal T}_0$ and $Y$ is a subordinator constructed in Theorem \ref{theorem_main}. Then $Y \in {\mathcal T}_0$, in particular 
the function $y \pi_Y(y)=p_X(y,y)$ is completely monotone. 
\end{proposition} 
\begin{proof}
We will need the following result (see Theorem 3.1.2 in \cite{Bondesson}): {\it Let $\eta$ be a positive random variable and 
define $f(z):=\ln \e\left[ {\rm e}^{-z \eta}\right]$. Then $\eta$ has a GGC distribution if and only if $f'(z)$ is a Pick function.}  

 Assume that $X \in {\mathcal T}_0$. According to the above result, $-\Phi_X'(z)$ is a Pick function. 
 Let $Y$ be a subordinator constructed from $X$ in Theorem \ref{theorem_main}. We recall that 
 $\phi(q)$ is defined as the solution to $z-\Phi_X(z)=q$ and $\Phi_Y(z)=\phi(z)-\phi(0)-z$. 
 Since $X\in {\mathcal T}_0$, it has a completely monotone L\'evy density, thus according to Proposition \ref{thm_compl_monotone}, the same is true for $Y$. Therefore, the three functions 
 $\Phi_X(z)$, $\Phi_Y(z)$ and $\phi(z)$ are Pick functions. 
 Taking derivative of the identity $\phi(q)-\Phi_X(\phi(q))=q$ we find that 
\beqq 
 -\phi'(q)=-\frac{1}{1-\Phi_X'(\phi(q))}.
\eeqq 
 Since the composition of Pick functions is also a Pick function, and since the three functions 
\beqq
F: q\mapsto \phi(q), \;\;\; G: z\mapsto -\Phi_X'(z), \;\;\; H: w\mapsto -\frac{1}{1+w}
\eeqq 
are Pick functions, we conclude that $-\phi'(q)=H(G(F(q)))$ is also a Pick function.  Therefore, $-\Phi_Y'(q)=-\phi'(q)+1$ is a Pick function, which implies $Y \in {\mathcal T}_0$. 
\end{proof}

\section{Examples}\label{sec_examples}

In this section we present several new families of subordinators possessing explicit transition semigroups. 
Our first two examples are related to the Lambert W-function \cite{Corless96,Corless2, Pakes}, and we will start by reviewing some of its properties. 
Lambert W-function is defined as the inverse to the function $w\in \c\mapsto w{\rm e}^w$. When $z\ne 0$, the equation
$w{\rm e}^w=z$ has infinitely many solutions, therefore we will have infinitely many branches of the Lambert W-function, which we will label by $W_{k}(z)$. See \cite{Corless96} for detailed discussion of branches of Lambert W-function. We will be only interested 
in two real branches of the Lambert W-function,  $W_0(z)$ (the principal branch) and $W_{-1}(z)$. For $z>-1/e$, these are defined 
as the {\it real} solutions to $w{\rm e}^w=z$. It is easy to show that the function $w{\rm e}^w$ is increasing for $w>-1$ and decreasing for $w<-1$, see figure \ref{fig1}. Therefore, for $z\ge 0$ there is a unique real solution, corresponding to $W_0(z)$, while for $-1/e<z<0$ there exist two real solutions $W_{-1}(z)<-1<W_0(z)<0$. The graphs of the two functions $W_0(z)$ and $W_{-1}(z)$ are presented on figures \ref{p2} and \ref{p3}. The function $W_0(z)$ is the principal branch of the Lambert W-function, and it has received considerably more attention compared to its other sibling, $W_{-1}(z)$. In many ways it is a simpler function, for example it
is 
a classical example for which the Lagrange inversion formula gives a very simple and explicit Taylor series at $z=0$ (see formula (3.1) in \cite{Corless96})
,
\beq\label{W_series}
W_0(z)=\sum\limits_{n\ge 1} (-n)^{n-1} \frac{z^n}{n!}, \;\;\; \vert z \vert<1/e.
\eeq

\label{wew_discussion}

\begin{figure}
\centering
\subfloat[][The function $z=w{\rm e}^w$]{\label{p1}\includegraphics[height =6cm]{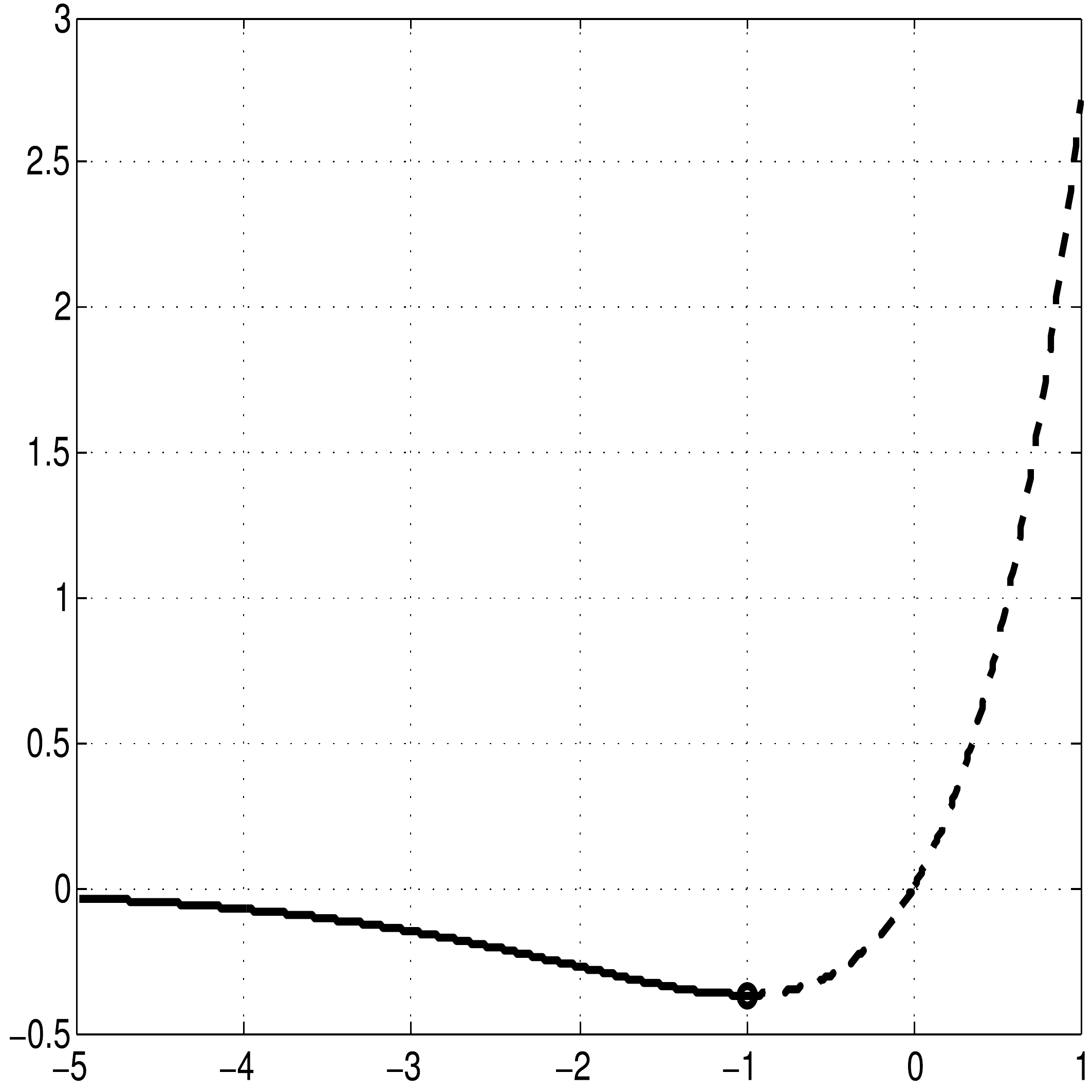}} 
\subfloat[][$W_0(z)$: the solution to $w{\rm e}^w=z$]{\label{p2}\includegraphics[height =6cm]{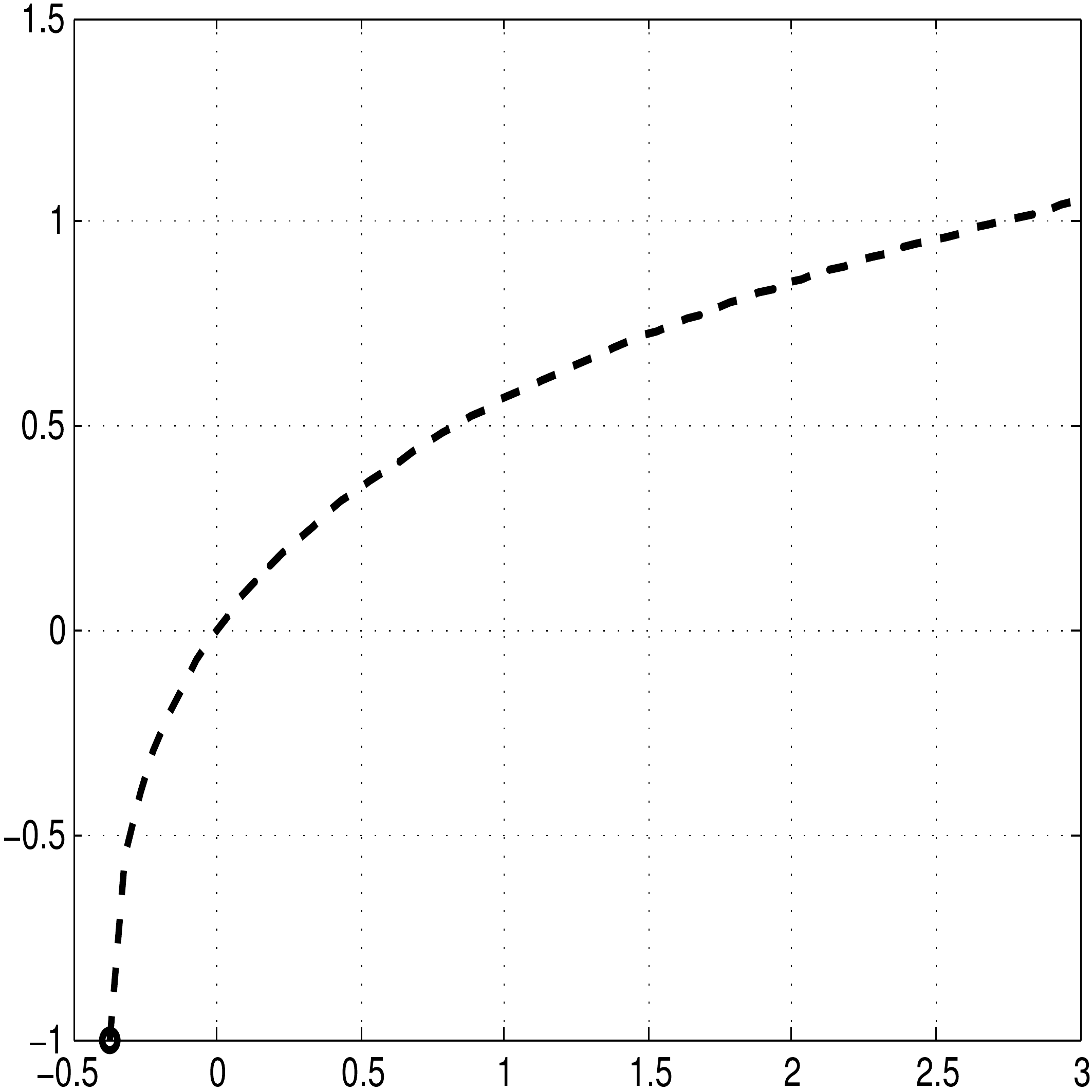}}
\subfloat[][$W_{-1}(z)$: the solution to $w{\rm e}^w=z$]{\label{p3}\includegraphics[height =6cm]{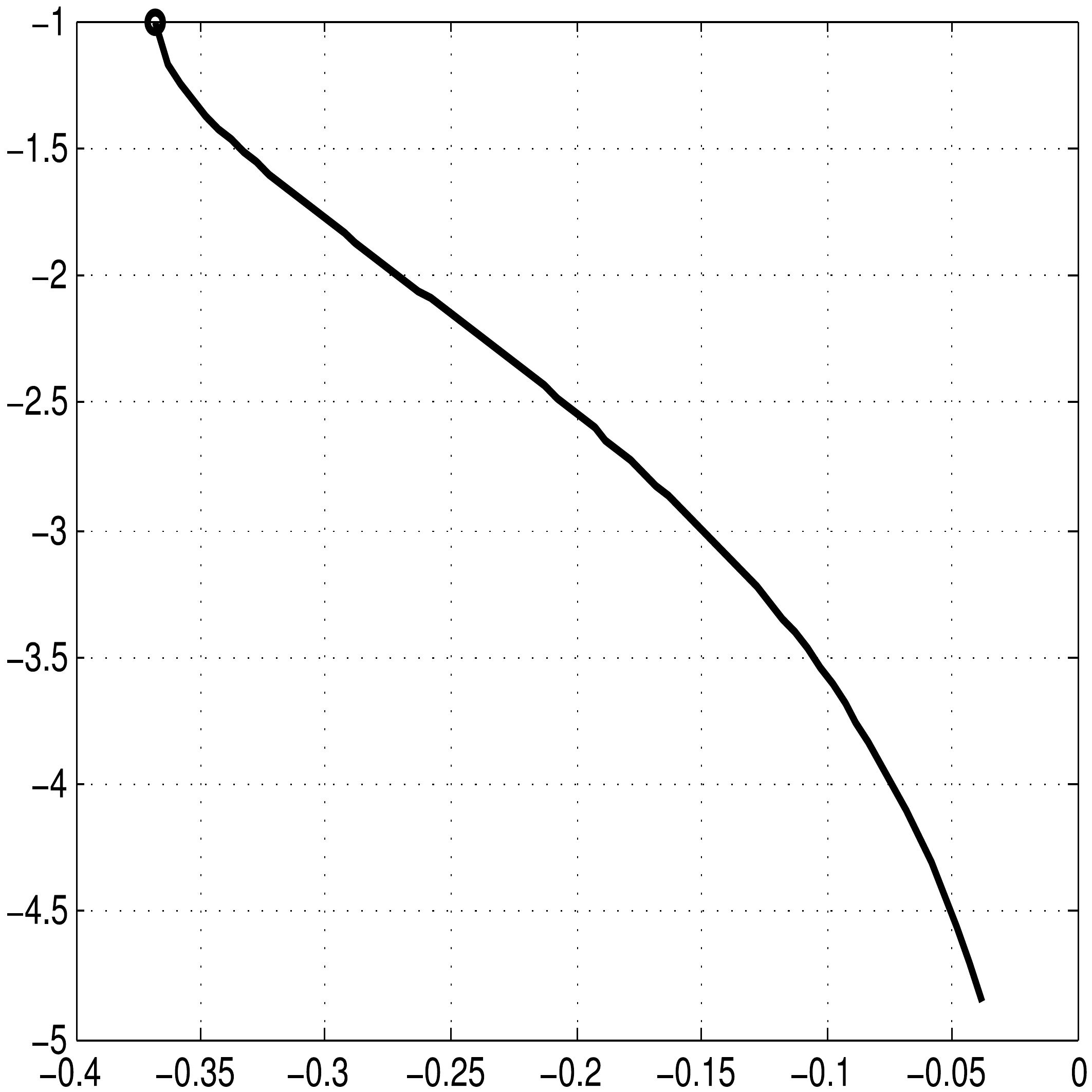}} 
\caption{The two real branches of the Lambert W-function: $W_0(z)$ is an increasing function which maps $[-1/e,\infty)$ onto $[-1,\infty)$, and
$W_{-1}(z)$ is a decreasing function which maps $[-1/e,0)$ onto $(-\infty,-1]$.} 
\label{fig1}
\end{figure}

\subsection{Poisson process}\label{subsec_Poisson}

In this section we construct a subordinator starting from the spectrally negative process $\xi_t=t-N_{ct}$, where $N$ is the standard Poisson process (i.e. with unit rate of arrival).

\begin{proposition}\label{prop_Poisson}
For $c>0$ the function $\Phi_Y(z)=W_0\left(-c{\rm e}^{-c-z}\right)-W_0\left(-c{\rm e}^{-c}\right)$ is the Laplace exponent of a compound Poisson process. The distribution of $Y_t$ is supported on $\{0,1,2,\cdots\}$ and is given by
\beq\label{gen_Poisson_distribution}
\p(Y_t=n)=ct\frac{(c(n+t))^{n-1}}{n!} {\rm e}^{-c(n+t)+at},  \;\;\; n\ge 0,
\eeq
where $a:=0$ if $c\le 1$ and $a:=c+W_0\left(-c{\rm e}^{-c} \right)$ if $c>1$.
The L\'evy measure is given by
\beq\label{Levy_measure_Poisson}
\Pi_Y(\{n\})=\frac{n^{n-1}}{n!} c^n {\rm e}^{-cn}, \;\;\; n\ge 1.
\eeq
\end{proposition}
\begin{proof}
Consider the spectrally negative L\'evy process $\xi_t=t-N_{ct}$, where $N$ is the standard Poisson process. Our goal is to compute the Laplace exponent, transition semigroup and the L\'evy measure of the subordinator $\{\tau_x^+\}_{x\geq 0}$. On account of the fact that the paths of $\xi$ are piecewise linear, it is easy to see that $\{\tau_x^+\}_{x\geq 0}$ is necessarily  a compound Poisson process. Moreover, as noted in the proof of Theorem \ref{theorem_main},  this subordinator must also have unit drift. Its  jump size distribution must also supported on positive integers. This is intuitively clear on account of the fact that if exactly $n$ jumps occur during an excursion of $\xi$ from its maximum, then, since each jump is of unit size and $\xi$ has a unit upward drift, then it requires precisely $n$ units of time to return to the maximum. This is also clear from the analytical relation (\ref{formula_PiY}).

In order to find the Laplace exponent $\phi(q)$ we need to solve the following equation 
\beqq 
z-c(1-{\rm e}^{-z})=q.
\eeqq
Changing variables $w=z-c-q$ we rewrite the above equation as 
\beqq
{\rm e}^{w}w=-c {\rm e}^{-c-q},
\eeqq
which gives us 
\beqq
z=\phi(q)= W\left(- c {\rm e}^{-c-q} \right) + c + q,
\eeqq
where $W$ is one of the two real branches of the Lambert W-function. We need to choose the correct branch of the Lambert W-function. 
Since $\phi(q)-q-\phi(0)$  and hence $\phi(q)-q$ is  the Laplace exponent of a subordinator,  it must be increasing in $q$. Since $W_0(z)$ is increasing while $W_{-1}(z)$ is decreasing, this shows that the correct branch is $W=W_0$. Therefore we conclude 
\beq\label{phi_Poisson}
\phi(q)=W_0\left(- c {\rm e}^{-c-q} \right) + c + q, \qquad q\geq 0. 
\eeq
Note that  $\{\tau_x^+\}_{x\geq 0}$ is killed at rate $\phi(0)=W_0\left(- c {\rm e}^{-c} \right) + c$ if $c>1$, and, otherwise, at rate $\phi(0)=0$ if $c\le 1$.

Next, let us find the transition semi-group of $\{\tau_x^+\}_{x\geq 0}$. As we have discussed above,  $\{\tau_x^+\}_{x\geq 0}$ has unit drift and its jump distribution is concentrated on the positive integers. This implies that the 
distribution of $\tau_x^+$ is supported on $\{x,x+1,x+2,\cdots\}$. Let us define $p_n(x)=\p(\tau_x^+=n+x)$. Then we find, for $t,y>0$,
\beqq
\int_y^{\infty} \p(\tau_x^+ \le t) \frac{\d x}{x}=\int_y^{\infty} \sum\limits_{n\ge 0} {\bf 1}_{\{n+x\leq t\}} p_n(x) \frac{\d x}{x}=
\sum\limits_{0\le n \le t-y} \int_y^{t-n} p_n(x) \frac{\d x}{x}. 
\eeqq
At the same time, 
\beqq
\int_0^t \p(\xi_s>y) \frac{\d s}{s}&=&\int_0^t \p(N_{cs}<s-y) \frac{\d s}{s}\\
&=&\int_0^t   \sum\limits_{n\ge 0} {\bf 1}_{\{n<s-y\}} \frac{(cs)^n}{n!} {\rm e}^{-cs} \frac{\d s}{s}
\\&=&\sum\limits_{0 \le n < t-y} \int_{n+y}^t  \frac{(cs)^n}{n!}{\rm e}^{-cs} \frac{\d s}{s}\\
&=&
\sum\limits_{0 \le n < t-y} \int_{y}^{t-n}  cs\frac{(c(s+n))^{n-1}}{n!}{\rm e}^{-c(s+n)} \frac{\d s}{s}. 
\eeqq
The above two equations combined with Kendall's identity \eqref{Kendalls_identity} give us
\beq\label{p_tau_x_Poisson}
\p(\tau_x^+ = n+x)=cx \frac{(c(n+x))^{n-1}}{n!} {\rm e}^{-c(n+x)}, \;\;\; n\ge 0.
\eeq
Now we define the subordinator $Y$, with zero drift coefficient and zero killing rate, via the Laplace exponent $\Phi_Y(z)=\phi(z)-z-\phi(0)$. The formula for the transition semigroup 
\eqref{gen_Poisson_distribution} follows from \eqref{p_tau_x_Poisson}. 
\end{proof}

When $c\in (0,1)$, the distribution given in \eqref{gen_Poisson_distribution} was introduced in 1973 by Consul and Jain \cite{Consul_Jain}, 
who called it the generalized Poisson distribution (see also \cite{Pakes}). Note that this distribution changes behavior at $c=1$. Using Stirling's approximation for 
$n!$ we find that 
\beqq
\Pi_Y(\{n\})=\frac{1}{\sqrt{2\pi}} n^{-\frac{3}{2}} {\rm e}^{-(c-1-\ln(c))n} \left(1+o(1)\right), \;\;\; n\to +\infty,
\eeqq
therefore the jump distribution of $Y$ has exponential tail when $c\ne 1$ and a power-law tail (with $\e[Y_1]=+\infty$) for $c=1$. 


\subsection{Gamma process}\label{subsec_gamma}

In this section we construct a subordinator using Theorem \ref{theorem_main} by  starting from a gamma subordinator. We recall that a gamma subordinator $X$ is defined by the Laplace exponent 
$\Phi_X(z)=c \ln(1+ \theta z)$, $z\geq 0$, where the constants $c,\theta>0$. It is well-known that $X$ has zero drift and that the transition probability density and the density of the L\'evy measure are given by
\beqq
p_X(t,x)=\frac{x^{ct-1} {\rm e}^{-\frac{x}{\theta}}}{\theta^{ct}\Gamma(ct)}, \;\;\;
\pi_X(x)=\frac{c}{x} {\rm e}^{-\frac{x}{\theta}}, \qquad x, t>0.
\eeqq

\begin{proposition}\label{prop_Gamma}
The function
\beq\label{phi_gamma}
\Phi_Y(z):=- c
 W_{-1}\left( -\frac{1}{\theta c}\exp\left(-\frac{1+\theta z}{\theta c} \right)\right)
 +c W_{-1}\left( -\frac{1}{\theta c}\exp\left(-\frac{1}{\theta c} \right)\right)-z,\qquad z\geq 0,
\eeq
is the Laplace exponent of a subordinator $Y \in {\mathcal T}_0$. 
The transition probability density of $Y$ is
\beq\label{Ressel}
p_Y(t,y)=\frac{c\theta^{-1} t}{\Gamma(1+c(t+y))} \left(\frac{y}{\theta} \right)^{c(t+y)-1} {\rm e}^{-\frac{y}{\theta}+at},\qquad y,t>0,
\eeq
where $a:=0$ if $\theta c \le 1$ and $a:=-1/\theta - c W_{-1} \left( -\frac{1}{\theta c} {\rm e}^{-\frac{1}{\theta c}} \right)$ if $\theta c>1$. 
The density of the L\'evy measure is given by
\beqq
\pi_Y(y)=\frac{c\theta^{-1}}{\Gamma(1+cy)} \left(\frac{y}{\theta} \right)^{cy-1} {\rm e}^{-\frac{y}{\theta}}, \qquad y>0.
\eeqq
\end{proposition}
\begin{proof}
This result is a straightforward application of Theorem \ref{theorem_main} and Proposition \ref{prop_Thorin_class}, we only need to identify the function $\phi(q)$, which is the solution 
to $z-c \ln(1+\theta z)=q$. Making change of variables
$u=-1/(\theta c) -z/c$ we can rewrite this equation as
\beqq
u{\rm e}^u = -\frac{1}{\theta c} {\rm e}^{-\frac{1}{\theta c} - \frac{q}{c}},
\eeqq
therefore
\beqq
u=W\left( -\frac{1}{\theta c} {\rm e}^{-\frac{1}{\theta c} - \frac{q}{c}} \right),
\eeqq
where $W$ is one of the two real branches of the Lambert W-function. Again, we need to choose the correct branch, $W_0$ or $W_{-1}$. Let us consider
\beqq
\phi(q)=-\frac{1}{\theta}-cu=-\frac{1}{\theta} - c W \left( -\frac{1}{\theta c} {\rm e}^{-\frac{1}{\theta c} - \frac{q}{c}} \right).
\eeqq
We know that $\phi(q)$ is the Laplace exponent of a subordinator with  drift rate equal to one, therefore  $\phi(q)$ is unbounded on $q\in (0,\infty)$. From the properties of $W_0$ and $W_{-1}$ (see 
figure \ref{fig1}) this is only possible if we choose the branch $W=W_{-1}$. Thus we obtain
\beqq
\phi(q)=-\frac{1}{\theta}-cu=-\frac{1}{\theta} - c W_{-1} \left( -\frac{1}{\theta c} {\rm e}^{-\frac{1}{\theta c} - \frac{q}{c}} \right).
\eeqq
Note that $\phi(0)=0$ if and only if $\theta c\le 1$. The rest of the proof follows from Theorem \ref{theorem_main}
and from Proposition \ref{prop_Thorin_class}.  
\end{proof}

The distribution given in \eqref{Ressel} goes back to Kendall \cite{Kendall1957}. It is also known as Ressel (or Kendall-Ressel) distribution (see \cite{Letac1990,Vinogradov2011}).  
Using Stirling's approximation for the Gamma function we find that 
\beqq
\pi_Y(y)=\sqrt{\frac{c}{2\pi}} y^{-\frac{3}{2}} {\rm e}^{-(\ln(\theta c)-1+\frac{1}{\theta c})cy} \left(1+o(1)\right), \;\;\; y\to +\infty,
\eeqq
therefore the L\'evy density of $Y$ has exponential tail when $\theta c\ne 1$ and a power-law tail (with $\e[Y_1]=+\infty$) for $\theta c=1$.

\subsection{Stable processes}\label{subsec_stable}

In this section, we obtain new families of subordinators which are related to stable processes. We define
\beq\label{p_series_alpha_01}
g(x;\alpha):=\frac{1}{\pi}\sum\limits_{n\ge 1} (-1)^{n-1} \frac{\Gamma(1+\alpha n)}{n!} \sin(\pi n \alpha) 
x^{-n \alpha -1}, \;\;\; x>0, \; 0<\alpha<1,
\eeq
and 
\beq\label{p_series_alpha_12}
g(x;\alpha):=\frac{1}{\pi}\sum\limits_{n\ge 1} (-1)^{n-1} \frac{\Gamma\left(1+n/\alpha\right)}{n!} \sin\left(\tfrac{\pi n}{\alpha}\right) 
x^{n-1}, \;\;\; x\in \r, \; \alpha>1. 
\eeq 
Note that, for $\alpha>1$, the function $x\mapsto g(x;\alpha)$ is entire and  satisfies the identity
\beq\label{Zolotarev_duality}
 x g(x;\alpha)=x^{-\alpha} g(x^{-\alpha};\alpha^{-1}), \;\;\; x>0, \alpha>1. 
\eeq
The function $g(x;\alpha)$ has the following probalistic interpretation: for $\alpha \in (0,1)$ \{resp. $\alpha \in (1,2)$\} it is the probability density function of a strictly stable random variable $U$ defined by $\e[\exp(-z U)]=\exp(-z^{\alpha})$
\{resp. $\e[\exp(z U)]=\exp(z^{\alpha})$\}, see Theorem 2.4.2 in \cite{Zolotarev1986}.  Identity \eqref{Zolotarev_duality} is just a special case of the so-called Zolotarev duality, see Theorem 2.3.2 in \cite{Zolotarev1986}. It is known that $U$ has a GGC distribution, see
example 3.2.1 in \cite{Bondesson}.

When $\alpha$ is a rational number, the function $g(x;\alpha)$ can be given in terms of hypergeometric functions, for example: 
\beqq
g(x;\tfrac{1}{3})=\frac{x^{-\frac{3}{2}}}{3\pi} K_{\frac{1}{3}}\left( \frac{2}{3\sqrt{3x}}\right), \;\;\;
g(x;\tfrac{2}{3})=\sqrt{\frac{3}{\pi}} x^{-1}{\rm e}^{-\frac{2}{27x^2}}W_{\frac{1}{2},\frac{1}{6}}\left(\frac{4}{27x^2}\right),\qquad x>0,
\eeqq
where $K_{\nu}(x)$ denotes the modified Bessel function of the  second type and $W_{a,b}(x)$ denotes the
Whittaker function (see \cite{Jeffrey2007}). The above two formulas can be found in \cite{Zolotarev1986} (see formula 2.8.31 and formula 2.8.33 with a slight normalizing correction $1/\sqrt{3\pi} \mapsto \sqrt{3/\pi}$).

\begin{proposition}\label{prop_stable1}
Assume that $\alpha\in (0,1)$ and $c>0$. For $q\ge 0$ define $\phi(q)$, $q\geq 0,$ as the unique positive solution to the equation $z-cz^{\alpha}=q$. Then 
the function $\Phi_Y(z)=\phi(z)-c^{\frac{1}{1-\alpha}}-z$ is the Laplace exponent of a subordinator $Y\in {\mathcal T}_0$.  
The transition probability density of the subordinator $Y$  is given by
\beq\label{p_y_stable1}
p_Y(t,y)= t\exp\left(c^{\frac{1}{1-\alpha}}t\right) \frac{(c(t+y))^{-\frac{1}{\alpha}}}{t+y} g\left(y(c(t+y))^{-\frac{1}{\alpha}};\alpha\right)\qquad x,t>0.
\eeq
The density of the L\'evy measure is given by
\beq\label{pi_y_stable1}
\pi_Y(y)= c^{-\frac{1}{\alpha}}y^{-\frac{1}{\alpha}-1} g\left(c^{-\frac{1}{\alpha}}y^{1-\frac{1}{\alpha}};\alpha\right),\qquad y>0.
\eeq
\end{proposition}
\begin{proof}
Let $X$ be an $\alpha$-stable subordinator, having Laplace exponent $\Phi_X(z)=cz^{\alpha}$.  
Due to the scaling property $a^{-\frac{1}{\alpha}}X_{at}\stackrel{d}{=} X_t$ we find that the density of
$X_t$ is given by $p_X(t,x)=g(x(ct)^{-\frac{1}{\alpha}};\alpha) (ct)^{-\frac{1}{\alpha}}$. 
The rest of the proof is a straightforward application of Theorem \ref{theorem_main}, Proposition \ref{prop_Thorin_class} and the fact that $\phi(0)=c^{\frac{1}{1-\alpha}}$. 
\end{proof}

\begin{remark}\label{mean}
We can also  compute the mean of the subordinator $Y$, but without having to consider the tail of the measure $\pi_Y$ as in the previous examples. Recall that $\phi(q)$ satisfies
$
\psi_\xi(\phi(q))= q
$, for $q\geq0$. Differentiating, it follows that, for $q>0$, $\phi'(q) \psi_\xi'(\phi(q))= 1$ and hence, 
\[
\mathbb{E}[Y_1] = \lim_{q\to0}\phi'(q) -1= \frac{1}{\psi'_\xi(\phi(0))}-1.
\]
It follows that the subordinator $Y$ has infinite mean if and only if $\psi'(\phi(0))=0$.
This happens if and only if  $\phi(0) = 0$ and $\psi'(0+) = 0$. When that $\psi_\xi(z) = z  -\Phi_X(z)$, $Y$ has infinite mean if and only if $\phi(0)=0$ and $\Phi_X'(0) =1$.
One easily shows in this example that 
\[
\mathbb{E}[Y_1]= \frac{1}{1 - c\alpha (c^{\frac{1}{1-\alpha}})^{\alpha -1}} - 1
=\frac{\alpha}{1-\alpha}.
\]
\end{remark}

In the next proposition, we use Theorem \ref{theorem_main} in combination with a choice of $\xi$ which is not the difference of a unit drift and a subordinator (and therefore a process of bounded variation). Instead we choose $\xi$ directly to be a spectrally negative stable process with unbounded variation added to a unit positive drift.

\begin{proposition}\label{prop_stable2}
Assume that $\alpha\in (1,2)$ and $c>0$. For $q\ge 0$ define $\Phi_Y(q)$ as the unique positive solution to the equation 
$z+cz^{\alpha}=q$. Then $\Phi_Y(q)$ is the Laplace exponent of an infinite mean subordinator $Y\in {\mathcal T}_0$. 
The transition probability density of the subordinator $Y$  is given by
\beq\label{p_y_stable2}
p_{Y}(t,y)= c^{-\frac{1}{\alpha}}t y^{-\frac{1}{\alpha}-1}  g\left((t-y) (cy)^{-\frac{1}{\alpha}};\alpha\right)\qquad y,t>0.
\eeq
The density of the L\'evy measure is given by
\beq\label{pi_y_stable2}
\pi_Y(y)= c^{-\frac{1}{\alpha}} y^{-\frac{1}{\alpha}-1}  g\left(-c^{-\frac{1}{\alpha}} y^{1-\frac{1}{\alpha}} ;\alpha\right),\qquad y>0.
\eeq
\end{proposition}
\begin{proof}
Let $\tilde \xi$ be a spectrally negative $\alpha$-stable process, defined by the Laplace exponent $\e[\exp(z \tilde \xi_1)]=\exp(cz^{\alpha})$, $z\geq 0$.  
Consider the spectrally negative process $\xi_t=\tilde\xi_t+t$.
The density of $\xi_t$ is 
\beqq
p_{\xi}(t,x)=(ct)^{-\frac{1}{\alpha}} g((x-t) (ct)^{-\frac{1}{\alpha}};\alpha), \qquad x\in\mathbb{R}, t>0.
\eeqq 
We define the subordinator $Y_t=\tau_t^+$, $t\geq 0$. 
Formula \eqref{p_y_stable2} follows from Kendall's identity \eqref{Kendalls_identity_v2} and formula \eqref{pi_y_stable2} follows from 
\eqref{formula_PiY}. Referring to the computations in Remark \ref{mean}, it is straightforward to see that $\mathbb{E}[Y_1]  =+\infty.$
 Let us prove that $Y \in {\mathcal T}_0$. The proof will follow the same path as the proof of Proposition
\ref{prop_Thorin_class}. Taking derivatives with respect to $q$  on both sides of the identity
\beqq
\Phi_Y(q)+c\Phi_Y(q)^{\alpha}=q
\eeqq
we find that
\beqq
-\Phi_Y'(q)=-\frac{1}{1+\alpha c \Phi_Y(q)^{\alpha-1}}.
\eeqq
According to Proposition \ref{thm_compl_monotone}, the function $\Phi_Y(q)$ is a Pick function, therefore $-\Phi_Y'(q)=H(G(F(q)))$ is a composition of the three Pick functions
\beqq
F: q \mapsto \Phi_Y(q), \;\;\; G: z \mapsto z^{\alpha-1}, \;\;\; H: w \mapsto -\frac{1}{1+\alpha c w}.
\eeqq
This shows that $-\Phi_Y'(q)$ is a Pick function, therefore $Y \in {\mathcal T}_0$. 
\end{proof}

\begin{remark}
The proof of Proposition \ref{prop_stable2} shows  that the subordinator $Y$ is the ascending ladder time subordinator
of an unbounded variation spectrally negative stable process with unit {\it positive} drift.
One could ask the following natural question: what if we consider the ascending ladder time subordinator of an unbounded variation spectrally negative stable process with unit {\it negative} drift, will we get a new family of subordinators? It turns out that in this case we would obtain (up to scaling) the same family of subordinators as in Proposition \ref{prop_stable1}. The details are left to the reader. The case that we choose $\xi$ to be just an unbounded variation spectrally negative stable process is uninteresting. In that case Theorem \ref{theorem_main} simply delivers the classical result that $Y$ is the ascending ladder time process which is a stable subordinator with index $1/\alpha$.
\end{remark}

\subsection{Bessel subordinator}\label{subsec_Bessel}

A Bessel subordinator $X$ is defined by the Laplace exponent
\beq\label{def_Bessel_Phi}
\Phi_X(z)=c \ln\left(1+\theta z+\sqrt{(1+\theta z)^2 -1} \right), \qquad z\geq 0,
\eeq
where $c>0$ and $\theta>0$. 
This process was introduced in \cite{MNY}, and it was shown that its transition density and the density of the L\'evy measure  are respectively  given by
\beqq
p_X(t,x)=c t x^{-1} {\rm e}^{-\frac{x}{\theta}} I_{ct}\left(\tfrac{x}{\theta}\right),
\;\;\; \pi_X(x)=c x^{-1} {\rm e}^{-\frac{x}{\theta}} I_{0}\left(\tfrac{x}{\theta}\right),\qquad t,x>0,
\eeqq
where $I_{\nu}(x)$ denotes the modified Bessel function of the first kind (see \cite{Jeffrey2007}). It is known 
that $X\in {\mathcal T}_0$, see example 1.6.b in \cite{james2008}. 
Applying Theorem \ref{theorem_main} and Proposition \ref{prop_Thorin_class}, as well as taking note of Remark \ref{mean}, we obtain the following result. 
\begin{proposition}\label{prop_Bessel}
For $q>0$ define $\phi(q)$ as the unique solution to the equation 
\beqq
z-c \ln\left(1+\theta z +\sqrt{(1+\theta z)^2 -1} \right)=q.
\eeqq
Then the function $\Phi_Y(z)=\phi(z)-\phi(0)-z$ is the Laplace exponent of a finite mean subordinator $Y \in {\mathcal T}_0$.  
The transition probability density of the subordinator $Y$  is given by
\beqq
p_Y(t,y)= c t y^{-1} {\rm e}^{\phi(0)t-\frac{y}{\theta}} I_{c(t+y)}\left(\frac{y}{\theta}\right).
\eeqq
The density of the L\'evy measure is given by
\beqq
\pi_Y(y)= c y^{-1} {\rm e}^{-\frac{y}{\theta}} I_{cy}\left(\frac{y}{\theta}\right).
\eeqq
\end{proposition}

\subsection{Geometric stable subordinator}\label{subsec_geom_stable}

Assume that $c>0$, $\theta>0$ and $\alpha \in (0,1)$. 
Consider a geometric stable subordinator $X$, which is defined by the Laplace exponent $\Phi_X(z)=c\ln(1+(\theta z)^{\alpha})$ (see \cite{Song_Vondracek} and \cite{Pillai}). This process can be constructed by taking an $\alpha$-stable subordinator and subordinating it with 
the Gamma process. The transition density and L\'evy density of $X$ are respectively given by
\beqq
p_X(t,x)=\frac{\alpha c t}{x} \sum\limits_{k\ge 0}  \frac{(-1)^k(1+ct)_k}{\Gamma(1+\alpha(ct+k)) k! }  \left(\frac{x}{\theta}\right)^{\alpha(ct+k)},
\;\;\; 
\pi_X(x)=
c\alpha x^{-1} E_{\alpha}\left(-\left(\tfrac{x}{\theta}\right)^{\alpha}\right),\qquad  t,x>0,
\eeqq
where $(a)_k:=a(a+1)\cdots (a+k-1)$ denotes the Pocchammer symbol  and
\beqq
E_{\alpha}(x):=\sum\limits_{k\ge 0}  \frac{x^k}{\Gamma(1+\alpha k)} 
\eeqq
denotes the Mittag-Leffler function
(see \cite{Song_Vondracek}). It is known that $x\pi_X(x)$ is a completely monotone function
(see \cite{Gorenflo}), thus $X\in {\mathcal T}_0$. 
Applying Theorem \ref{theorem_main} and Proposition \ref{prop_Thorin_class}, and again making use of Remark \ref{mean}, we obtain the following family of subordinators. 
\begin{proposition}\label{prop_geom_stable}
Assume that $c>0$, $\theta>0$ and $\alpha \in (0,1)$.  For $q>0$ define $\phi(q)$ as the unique solution to the equation 
\beqq
z- c\ln\left(1+(\theta z)^{\alpha}\right)=q. 
\eeqq
Then the function $\Phi_Y(z)=\phi(z)-\phi(0)-z$ is the Laplace exponent of a finite mean subordinator $Y \in {\mathcal T}_0$.  
The transition probability density of the subordinator $Y$  is given by
\[
p_Y(t,y)= {\rm e}^{\phi(0)t} \frac{\alpha c t}{y} 
\sum\limits_{k\ge 0}  \frac{(-1)^k(1+c(t+y))_k}{\Gamma(1+\alpha(c(t+y)+k)) k! }  \left(\frac{y}{\theta}\right)^{\alpha(c(t+y)+k)}, \qquad y,t>0.
\]
The density of the L\'evy measure is given by
\[
\pi_Y(y)= \frac{\alpha c}{y}
\sum\limits_{k\ge 0}  \frac{(-1)^k(1+c y)_k}{\Gamma(1+\alpha(c y+k)) k! }  \left(\frac{y}{\theta}\right)^{\alpha(c y+k)}, \qquad y>0.
\]
\end{proposition}

\subsection{Inverse Gaussian subordinator}\label{subsec_IG}

If we consider an inverse Gaussian subordinator $X$, having Laplace exponent 
$\Phi_X(z)=c(\sqrt{1+\theta z}-1)$, then it is easy to see that the subordinator $Y_t=\tau_t^+$, constructed from 
$X$ via Theorem \ref{theorem_main}, is also in the class of inverse Gaussian subordinators. This is not surprising, since the inverse Gaussian subordinator itself appears as the first passage time 
of the Brownian motion with drift, and one can show that applying this construction repeatedly does not produce new families of subordinators (see the discussion on page  \pageref{discussion_repeated_Kendall}).

\section{Applications}\label{sec_applications}

The results that we have obtained in the previous sections have interesting and non-trivial implications for Analysis and Special Functions. Every family of subordinators that we have discussed above leads to an explicit Laplace transform identity of the form
\beq\label{explicit_Laplace}
\int_0^{\infty} {\rm e}^{-zy} \p(Y_t \in \d y) ={\rm e}^{-t \Phi_Y(z)},\qquad z\geq 0,
\eeq
and it seems that in all of these cases (except for the first example involving Poisson process) we obtain new Laplace transform identities. We do not know of a simple direct analytical proof of these results (we have found one way to prove them, but this method is just a complex-analytical counterpart of the original probabilistic proof of Kendall's identity). 

Below we present a number of analytical statements that follow from our results in Section \ref{sec_examples}. 

\vspace{0.25cm}
\noindent
{\bf Example 1:} For $r<0$ and $t \in (0,{\rm e}^{-1})$ 
\beq\label{W_formula1}
\left(\frac{W_{-1}(-t)}{-t}\right)^{r}={\rm e}^{-rW_{-1}(-t)}=-\int\limits_{-r}^{\infty} r\frac{(w+r)^{w-1}}{\Gamma(1+w)} t^{w} \d w.
\eeq
This formula seems to be new, and it is a direct analogue of the known result 
\beq\label{W_formula2}
\left( \frac{W_0(-z)}{-z}\right)^r={\rm e}^{-r W_0(-z)}=
\sum\limits_{n\ge 0} r \frac{(n+r)^{n-1}}{n!} z^n, \;\;\; r\in \c, \; \vert z \vert<1/e,
\eeq
which can be found in \cite{Corless2}. 
Formula \eqref{W_formula2} can be obtained in two ways. The first 
one 
is the classical analytical approach 
via Lagrange inversion theorem (see 
\cite{Corless2}). The second approach is via proposition \ref{prop_Poisson} and \eqref{explicit_Laplace}. This example seems to indicate that when the subordinator $X$ in Theorem \ref{theorem_main} has support on the lattice, 
then Kendall's identity is an analytical statement which is equivalent to Lagrange inversion formula. 
Formula \eqref{W_formula1} is obtained in a similar way from Proposition \eqref{prop_Gamma}, and we hypothesize that in the general case  Kendall's identity can be considered as an integral analogue of Lagrange inversion formula. 
\vspace{0.25cm}

\noindent
{\bf Example 2:} Proposition \ref{prop_stable1} and \eqref{explicit_Laplace} give us the following resut: For $q>0$ we have
\beq
\int\limits_0^{\infty}  \sqrt{\frac{t+y}{y^3}} K_{\frac{1}{3}} \left(\frac{2}{3}\sqrt{\frac{(t+y)^3}{3y}} \right) {\rm e}^{-qy} \d y  =\frac{3\pi}{t} {\rm e}^{t(q-\phi(q))},
\eeq
where $\phi(q)$ is the solution to $z-z^{\frac{1}{3}}=q$.
\vspace{0.25cm}

\noindent
{\bf Example 3:} Proposition \ref{prop_stable1} and \eqref{explicit_Laplace} give us the following resut: For $q>0$ we have
\beq
\int\limits_0^{\infty}  \frac{{\rm e}^{-\frac{2}{27} \frac{(t+y)^3}{y^2}}}{y(t+y)} W_{\frac{1}{2},\frac{1}{6}} 
\left(\frac{4}{27}\frac{(t+y)^3}{y^2} \right) {\rm e}^{-qy} \d y  =\sqrt{\frac{\pi}{3}}\frac{1}{t} 
 {\rm e}^{t(q-\phi(q))},
\eeq
where $\phi(q)$ is the solution to $z-z^{\frac{2}{3}}=q$.
\vspace{0.25cm}

\noindent
{\bf Example 4:} From formula \eqref{Zolotarev_duality} we find that
\beqq
g(x;\tfrac{3}{2})=x^{-\frac{5}{2}}g(x^{-\frac{3}{2}};\tfrac{2}{3})= \sqrt{\frac{3}{\pi}}x^{-1} 
{\rm e}^{-\frac{2}{27}x^3}W_{\frac{1}{2},\frac{1}{6}}\left(\frac{4}{27}x^3\right).
\eeqq
Then Proposition \ref{prop_stable2} and \eqref{explicit_Laplace} give us the following resut: For $q>0$ we have
\beq
\int\limits_0^{\infty}  \frac{{\rm e}^{-\frac{2}{27} \frac{(t-y)^3}{y^2}}}{y(t-y)} W_{\frac{1}{2},\frac{1}{6}} 
\left(\frac{4}{27}\frac{(t-y)^3}{y^2} \right) {\rm e}^{-qy} \d y  =\sqrt{\frac{\pi}{3}}\frac{1}{t} {\rm e}^{-t\phi(q)},
\eeq
where $\phi(q)$ is the solution to $z+z^{\frac{3}{2}}=q$.
\vspace{0.25cm}

\noindent
{\bf Example 5:} Proposition \ref{prop_Bessel} and \eqref{explicit_Laplace} give us the following resut: For $q>0$, $c>0$ we have
\beq
\int\limits_0^{\infty} {\rm e}^{-y(\frac{1}{\theta}+q)}
I_{c(t+y)}\left(\frac{y}{\theta}\right)\frac{\d y}{y}=\frac{1}{ct} {\rm e}^{t(q-\phi(q))},
\eeq
where $\phi(q)$ is the solution to $z-c \ln\left(1+\theta z +\sqrt{(1+\theta z)^2 -1} \right)=q$.
\vspace{0.25cm}

\noindent
{\bf Example 6:} We recall that a subordinator $X$ belongs to the Thorin class ${\mathcal T}_0$ if and only if $x\pi_X(x)$ is a completely monotone function (where $\pi_X(x)$ is the L\'evy density of $X$). The fact that subordinators constructed in Propositions \ref{prop_Gamma},
\ref{prop_stable1}, \ref{prop_stable2}, \ref{prop_Bessel} and \ref{prop_geom_stable} belong to the class ${\mathcal T}_0$ implies that the following functions 
\beqq
f_1(y)&=&\frac{y^{cy}{\rm e}^{-y}}{\Gamma(1+cy)}, \;\;\; c>0, \;  y>0,  \\
f_2(y)&=&y^{-\frac{1}{\alpha}} g(y^{1-\frac{1}{\alpha}};\alpha), \;\;\; \alpha \in (0,1), \; y>0, \\
f_3(y)&=&y^{-\frac{1}{\alpha}} g(-y^{1-\frac{1}{\alpha}};\alpha), \;\;\; \alpha \in (1,2), \; y>0, \\
f_4(y)&=&{\rm e}^{-y} I_{cy}(y), \;\;\; c>0, \; y>0, \\
f_5(y)&=&\sum\limits_{k\ge 0}  \frac{(-1)^k(1+c y)_k}{\Gamma(1+\alpha(c y+k)) k! }  y^{\alpha(c y+k)}, \;\;\; c>0, \;
\alpha \in (0,1), \; y>0, 
\eeqq
are completely monotone. We are not aware of any simple analytical proof of this result.

\vspace{0.25cm}

\paragraph{Acknowledgements}
A. Kuznetsov acknowledges the support by the
Natural Sciences and Engineering Research Council of Canada. 
M.~Kwa\'snicki was supported by Polish National Science Centre (NCN) grant no. 2011/03/D/ST1/00311.
We would like to thank Takahiro Hasebe for providing many helpful comments on the paper, for pointing out the connection with GGC distributions and for proving Proposition \ref{prop_Thorin_class}. A. E. Kyprianou would like to thank Victor Rivero and Jean Bertoin for discussion.  We  would also like to thank V. Vinogradov for pointing out a number of important references from the statistics literature which escaped our  attention.


\end{document}